\documentclass{amsart}

\usepackage{amsfonts,amsmath,amssymb,amsthm,mathrsfs,amsxtra}
\usepackage{enumerate,verbatim}
\usepackage{lineno} 
\usepackage[all,2cell,ps]{xy}
\usepackage[pagebackref]{hyperref}

\DeclareMathOperator{\codim}{codim}
\DeclareMathOperator{\cx}{cx}

\DeclareMathOperator{\depth}{depth}
\DeclareMathOperator{\EI}{EI}

\DeclareMathOperator{\EP}{EP}
\DeclareMathOperator{\Ext}{Ext}

\DeclareMathOperator{\Hom}{Hom}

\DeclareMathOperator{\injdim}{injdim}
\DeclareMathOperator{\Mod}{mod}

\DeclareMathOperator{\projdim}{projdim}
\DeclareMathOperator{\rank}{rank}
\DeclareMathOperator{\Spec}{Spec}

\DeclareMathOperator{\Tor}{Tor}

\renewcommand{\ge}{\geqslant}
\renewcommand{\le}{\leqslant}

\newcommand{\fm}{\mathfrak{m}}
\newcommand{\fp}{\mathfrak{p}}
\newcommand{\fq}{\mathfrak{q}}

\newcommand{\lra}{\longrightarrow}

\newcommand{\oM}{\overline{M}}
\newcommand{\oN}{\overline{N}}

\newcommand{\oR}{\overline{R}}

\theoremstyle{plain}
\newtheorem{theorem}{Theorem}[section]
\newtheorem{lemma}[theorem]{Lemma}
\newtheorem{proposition}[theorem]{Proposition}
\newtheorem{corollary}[theorem]{Corollary}

\theoremstyle{definition}

\newtheorem{example}[theorem]{Example}

\newtheorem{hypothesis}[theorem]{Hypothesis}

\newtheorem{para}[theorem]{}

\theoremstyle{remark}
\newtheorem{remark}[theorem]{Remark}

\numberwithin{equation}{theorem}

\title[Vanishing of Exts and Tors]{Vanishing of (co)homology over deformations of Cohen-Macaulay local rings of minimal multiplicity}
\date{January 15, 2018}
\author[Dipankar Ghosh]{Dipankar Ghosh}
\address{Chennai Mathematical Institute, H1, SIPCOT IT Park, Siruseri, Kelambakkam, Chennai 603103, Tamil Nadu, India}
\email{dghosh@cmi.ac.in}
\thanks{Ghosh was supported by DST, Government of India under the DST-INSPIRE Faculty Scheme}

\author[Tony J. Puthenpurakal]{Tony J. Puthenpurakal}
\address{Department of Mathematics, Indian Institute of Technology Bombay, Powai, Mumbai 400076, India}
\email{tputhen@math.iitb.ac.in}


\begin{document}


\begin{abstract}
	Let $ R $ be a $ d $-dimensional Cohen-Macaulay (CM) local ring of minimal multiplicity. Set $ S := R/({\bf f}) $, where $ {\bf f} := f_1,\ldots,f_c $ is an $ R $-regular sequence. Suppose $ M $ and $ N $ are maximal CM $ S $-modules. It is shown that if $ \Ext_S^i(M,N) = 0 $ for some $ (d+c+1) $ consecutive values of $ i \ge 2 $, then $ \Ext_S^i(M,N) = 0 $ for all $ i \ge 1 $. Moreover, if this holds true, then either $ \projdim_R(M) $ or $ \injdim_R(N) $ is finite. In addition, a counterpart of this result for Tor-modules is provided. Furthermore, we give a number of necessary and sufficient conditions for a CM local ring of minimal multiplicity to be regular or Gorenstein. These conditions are based on vanishing of certain Exts or Tors involving homomorphic images of syzygy modules of the residue field.
\end{abstract}

\maketitle

\section{Introduction }\label{sec:introduction}

Throughout this article, unless otherwise specified, all rings are assumed to be commutative Noetherian local rings, and all modules are assumed to be finitely generated. Let $(R,\fm,k)$ be a Noetherian local ring. A celebrated result by Auslander and  Lichtenbaum, \cite[Corollary~2.2]{Aus61} and \cite[Corollary~1]{Lic66}, is the following:

\begin{theorem}[Rigidity Theorem]\label{thm:rigidity}
 	Let $ R $ be a regular local ring. For $ R $-modules $ M $ and $ N $, if $ \Tor_i^R(M,N) = 0 $ for some $ i \ge 1 $, then $ \Tor_j^R(M,N) = 0 $ for all $ j \ge i $.
 \end{theorem}
 
 Heitmann \cite{Hei93} showed that rigidity may fail even when $ R $ is a Cohen-Macaulay (CM) local ring and $ \projdim_R(M) $ is finite. Let $ S $ be a local complete intersection ring of codimension $ c $. In \cite[Theorem~1.6]{Mur63}, Murthy showed that $ c + 1 $ consecutive vanishing of Tors involving a pair of $ S $-modules $ M $ and $ N $ forces the vanishing of all subsequent Tors. We refer the reader to \cite[Theorem~9.3.6]{Avr98} for a concise proof of this result. Theorem~\ref{thm:rigidity} has been generalized further by Avramov and Buchweitz in \cite[Theorem~4.9]{AB00}. They improved the number $ c+1 $ of consecutive vanishing of Tors by replacing it by $ \cx_S(M) + 1 $, where $ \cx_S(M) \; (\le c) $ is the complexity of $ M $; see \ref{P}. Moreover, they proved a counterpart of this result for Ext-modules; see \cite[Theorem~4.7]{AB00}. In this article, we prove analogues of Murthy's result and that of Avramov and Buchweitz for deformations of CM local rings of minimal multiplicity.
 
 The {\it multiplicity} of an $ R $-module $ M $, i.e., the normalized leading coefficient of the Hilbert-Samuel polynomial $ P_M(n) $ ($ = $ length of $ M/\fm^{n+1}M $ for all sufficiently large $ n $) is denoted by $ e(\fm, M) $, or simply by $ e(M) $. In \cite[(1)]{Abh67}, Abhyankar showed that if $ R $ is CM, then $ e(R) \ge \mu(\fm) - \dim(R) + 1 $, where $ \mu(\fm) $ denotes the minimal number of generators of $ \fm $. If equality holds, then $ R $ is said to have {\it minimal multiplicity}, or {\it maximal embedding dimension}. It is well-known that if the residue field $ k $ is infinite, then $ R $ has minimal multiplicity if and only if there exists an $ R $-regular sequence $ \underline{x} $ with the property that $ \fm^2 = (\underline{x})\fm $; see, e.g., \cite[4.6.14(c)]{BH98}. Hence every regular local ring has minimal multiplicity. But the converse is not necessarily true, e.g., $ R_1 = k[U,V]/(U^2, UV, V^2) $ and $ R_2 = k[[U,V]]/(UV) $, where $ U $ and $ V $ are indeterminates, and $ k $ is a field. Note that $ R_1 $ is not even Gorenstein. 
 
 We now state our main results. We first give a result on the vanishing of Ext.

 \begin{theorem}[Theorem~\ref{thm:vanshing-ext-over-deformation}]\label{main-Th1}
 	Let $ R $ be a $ d $-dimensional CM local ring of minimal multiplicity. Set $ S := R/(f_1,\ldots,f_c) $, where $ f_1,\ldots,f_c $ is an $ R $-regular sequence. Let $ M $ and $ N $ be maximal Cohen-Macaulay {\rm (}MCM{\rm )} $ S $-modules. Then the following statements are equivalent:
 	\begin{enumerate}[{\rm (i)}]
 		\item $ \Ext_S^i(M,N) = 0 $ for some $ (d+c+1) $ consecutive values of $ i \ge 2 $.
 		\item $ \Ext_S^i(M,N) = 0 $ for all $ i \ge 1 $.
 	\end{enumerate}
 	Moreover, if this holds true, then $ \projdim_R(M) $ or $ \injdim_R(N) $ is finite.
 \end{theorem}
 
 Next we state our result on the vanishing of Tor:
 
 \begin{theorem}[Theorem~\ref{thm:vanshing-tor-over-deformation}]\label{main-Th2}
 	With the hypotheses as in {\rm Theorem~\ref{main-Th1}}, the following statements are equivalent:
 	\begin{enumerate}[{\rm (i)}]
 		\item $ \Tor^S_i(M,N) = 0 $ for some $ (d+c+1) $ consecutive values of $ i \ge c + 2 $.
 		\item $ \Tor^S_i(M,N) = 0 $ for all $ i \ge c+ 1 $.
 	\end{enumerate}
 	Moreover, if this holds true, then $ \projdim_R(M) $ or $ \projdim_R(N) $ is finite.
 \end{theorem} 
 
 We note that practically all results on complete intersection rings (including the results of Murthy, Avramov-Buchweitz) do use the fact that projective dimension of all modules over a regular local ring is finite. This fact is not necessarily true over rings of minimal multiplicity. The essential property of rings of minimal multiplicity that we use is the following: The first (and hence all subsequent) syzygy of a non-free MCM module is Ulrich. (An MCM module is assumed to be non-zero). Recall that an $ R $-module $ M $ is said to be {\it Ulrich} if $ M $ is an MCM $ R $-module and $ e(M) = \mu(M) $. It should be noted that for an MCM $ R $-module $ M $, we always have $ e(M) \ge \mu(M) $. Moreover, when $ k $ is infinite, then equality holds if and only if $ \fm M = (\underline{x}) M $ for some $ M $-regular sequence $ \underline{x} $; see \cite[Lemma~(1.3)]{BHU87}. We refer the reader to \cite{BHU87, HUB91} for more details on Ulrich modules.
 
 As an application of our result, we show that the commutative version of a conjecture of Tachikawa holds true for deformations of CM local rings of minimal multiplicity; see Theorem~\ref{Tachikawa}. As other applications, we obtain a few necessary and sufficient conditions for a deformation of a CM local ring of minimal multiplicity to be regular or Gorenstein. These conditions are based on the vanishing of certain Exts or Tors involving homomorphic images of syzygy modules of the residue field; see Theorems~\ref{app2} and \ref{app3}. Similar criteria for a CM local ring of minimal multiplicity to be regular or Gorenstein are given in Propositions~\ref{prop:min-mult-charc-reg-hom-image}, \ref{prop:min-mult-charc-Gor-hom-image} and Theorems~\ref{thm:min-mult-charc-reg-direct-summand}, \ref{thm:min-mult-charc-Gor-direct-summand}. These criteria are motivated by the following results: \cite[Corollary~1.3]{Dut89}, \cite[Proposition~7]{Mar96}, \cite[Corollary~9]{Avr96}, \cite[4.3 and 6.5]{Tak06}, \cite[3.2, 3.4 and 3.7]{GGP} and \cite[4.1, 5.1 and 5.5]{Gho}.
%
 
 Here is an overview of the contents of the article. In Section~\ref{sec:preliminaries}, we introduce some notations and discuss a few results that we need. In Section~\ref{sec:Ulrich-mod-is-test-mod}, we show properties of Ulrich modules as test modules for projective and injective dimensions. In Section~\ref{sec:vanishing-ext-tor}, we provide some results on the vanishing of Exts and Tors over CM local rings of minimal multiplicity. These are the base cases of Theorems~\ref{main-Th1} and \ref{main-Th2}. In Section~\ref{sec:criteria-reg-Gor-via-syzygy}, we give our results on regularity and Gorenstein properties of CM local rings of minimal multiplicity. In Section~\ref{sec:vanishing-ext-tor-deformation}, we prove our main results: Theorems~\ref{main-Th1} and \ref{main-Th2}. Finally, in Section~\ref{sec:application}, we give applications of our results.
 
\section{Preliminaries}\label{sec:preliminaries}

Throughout the article, $ R $ always denotes a CM local ring of dimension $ d $ with the unique maximal ideal $\mathfrak{m}$ and residue field $k$. For an $ R $-module $ M $, and $ n \ge 0 $, we denote the $n$th {\it syzygy module} of $M$ by $\Omega_n^R(M)$, i.e., the image of the $ n $th differential of an augmented minimal free resolution of $ M $. 
 
 \begin{para}\label{para:res-field-infinite}
 	To prove our results, we may without loss of generality assume that the residue field $k$ is infinite. If the residue field $ k $ is finite, then we use the standard trick to replace $ R $ by $ R' := R[X]_{\mathfrak{m}R[X]} $, where $ X $ is an indeterminate. Clearly, the residue field of $ R' $ is $ k(X) $, which is infinite. For more detail explanations, we refer the reader to \cite[Section~2.1]{Gho}.
 \end{para}

 \begin{para}
 	Let $ M $ be an $ R $-module, and $ x $ be an $ M $-regular element. It is not always true that $ e(\fm,M) = e(\fm/(x), M/xM) $. This holds true if $ x $ is an $ M $-superficial element. An element $ x \in \fm $ is called $ M $-{\it superficial} if there exists an integer $ c \ge 1 $ such that
 	\[
	 	\left( \fm^{n+1} M :_M x \right) \cap \fm^c M = \fm^n M \quad \mbox{for all } n \ge c.
 	\]
 	It is well-known that if $ k $ is infinite, then there exists an $ M $-superficial element.
 	If $ \dim(M) \ge 1 $, then for every $ M $-superficial element $ x $, it can be shown
 	that $ x \notin \fm^2 $, which yields that $ \mu(\fm/(x)) = \mu(\fm) - 1 $. 
 	If $ \depth(M) \ge 1 $, then one can easily show that every $ M $-superficial element is $ M $-regular;
 	see, e.g., \cite[p.~67, paragraph~3]{HM97} for the case $ M = R $.
 	Moreover, if $ x \in R $ is both $ M $-superficial and $ M $-regular, 
 	then $ e(\fm,M) = e(\fm/(x), M/xM) $; see \cite[Corollary~10(5)]{Pu1}. Thus, in view of these results, we obtain the following:
 \end{para}
 
 \begin{lemma}\label{lem:min-mult-mod-x}
 	\begin{enumerate}[{\rm (i)}]
 		\item
 		Assume that $ x \in R $ is both $ R $-superficial and $ R $-regular. If $ R $ has minimal multiplicity, then $ R/(x) $ also has minimal multiplicity.
 		\item
 		Let $ M $ be an $ R $-module. Assume that $ x \in R $ is both $ M $-superficial and $ R \oplus M $-regular. If $ M $ is Ulrich, then so is the $ R/(x) $-module $ M/xM $.
 	\end{enumerate}
 \end{lemma}
 
 We recall the following lemma concerning the behaviour of consecutive vanishing of Exts or Tors after going modulo a regular element.

\begin{lemma}\cite[p.~140, Lemma~2]{Mat86}\label{lem:Ext-Tor-mod-x}
	Suppose $ M $ and $ N $ are $ R $-modules. Let $ x $ be an $ R \oplus M \oplus N$-regular element. Set $\overline{(-)} := (-) \otimes_R R/(x)$. Fix two positive integers $ m $ and $ n $. If $ \Ext_R^i(M,N) = 0 $ {\rm(}resp. $ \Tor_i^R(M,N) = 0 ${\rm )} for all $ n \le i \le n + m $, then
	\begin{align*}
	&\Ext_{\oR}^i(\oM,\oN) = 0 \quad\mbox{for all } n \le i \le n + m - 1\\
	(\mbox{resp.} \quad &\Tor_i^{\oR}(\oM,\oN) = 0 \quad\mbox{for all } n + 1 \le i \le n + m).
	\end{align*}
\end{lemma}
 
 By considering the long exact sequences of Ext (resp. Tor) modules, and using induction on $ j $, one can prove the following:
 
\begin{lemma}\label{lem:syz-ext-tor}
	For $ R $-modules $ M $ and $ N $, we have the following isomorphisms:
	\begin{enumerate}[{\rm (i)}]
		\item $ \Ext_R^i\big(\Omega_j^R(M),N\big) \cong \Ext_R^{i+j}(M,N) \quad \mbox{for all }i \ge 1$ and $ j \ge 0 $.
		\item $ \Tor_i^R\big(\Omega_j^R(M),N\big) \cong \Tor_{i+j}^R(M,N) \quad \mbox{for all }i \ge 1$ and $ j \ge 0 $.
	\end{enumerate}
\end{lemma}

 Using a standard change of rings spectral sequence, we obtain the following:
 
\begin{lemma}\cite[Theorem~10.75]{Rot09}\label{lem:long-exact-seq-ext}
	Set $ S := R/(f) $, where $ f $ is an $ R $-regular element. Let $ M $ and $ N $ be $ S $-modules. Then we have the following long exact sequence:
	\begin{align*}
		0 \lra  \Ext_S^1(M,N) \lra & \Ext_R^1(M,N) \lra \Ext_S^0(M,N) \lra \\
		& \quad \quad \vdots \\
		\Ext_S^i(M,N) \lra & \Ext_R^i(M,N) \lra \Ext_S^{i-1}(M,N) \lra \\
		\Ext_S^{i+1}(M,N) \lra & \Ext_R^{i+1}(M,N) \lra \Ext_S^i(M,N) \lra \cdots.
	\end{align*}
\end{lemma}
 
 The following is the counterpart of Lemma~\ref{lem:long-exact-seq-ext} for Tor-modules.
 
\begin{lemma}\cite[Theorem~10.73]{Rot09}\label{lem:long-exact-seq-tor}
	Set $ S := R/(f) $, where $ f $ is an $ R $-regular element. Let $ M $ and $ N $ be $ S $-modules. Then we have the following long exact sequence:
	\begin{align*}
		\cdots \lra  \Tor_{i}^S(M,N) \lra & \Tor_{i+1}^{R}(M,N) \lra \Tor_{i+1}^S(M,N) \lra \\
		\Tor_{i-1}^S(M,N) \lra & \Tor_{i}^{R}(M,N) \lra \Tor_{i}^S(M,N) \lra \\
		& \quad \quad \vdots \\
		\Tor_0^S(M,N) \lra & \Tor_1^{R}(M,N) \lra \Tor_1^S(M,N) \lra 0.
	\end{align*}
\end{lemma}

Here we collect a few well-known facts about canonical modules for later use.

\begin{proposition}\label{prop:basic-facts-can-mod}
	Let $ R $ be a CM local ring.
	
	{\rm (i)} Let $ R $ be complete. Then $ R $ has a canonical module $ \omega_R $ {\rm (}cf. \cite[3.3.8]{BH98}{\rm )}. Moreover, every MCM $ R $-module $ M $ of finite injective dimension can be expressed as $ M \cong \omega_R^r $ for some $ r \ge 1 $; see \cite[Corollary~21.14]{Eis95}.
	
	{\rm (ii)} Let $ M $ be a CM $ R $-module, and $ \omega_R $ be a canonical module of $ R $. Then $ \Ext_R^i(M,\omega_R) = 0 $ for all $ i \neq \dim(R) - \dim(M) $; see, e.g., \cite[3.3.10]{BH98}.
	
	{\rm (iii)} Set $ R' := R/(f_1,\ldots,f_c) $, where $ f_1,\ldots,f_c $ is an $ R $-regular sequence. Suppose that $ R $ has a canonical module $ \omega_R $. Then $ R' $ also has a canonical module $ \omega_{R'} $, and $ \omega_{R'} \cong \omega_R/(f_1,\ldots,f_c)\omega_R $ {\rm (}cf. \cite[3.3.5(a)]{BH98}{\rm )}. Note that $ \injdim_R(\omega_R) $ is finite {\rm (}by definition of canonical modules{\rm )}. Using induction on $ c $, one can prove that $ \injdim_R(\omega_{R'}) $ is finite.
\end{proposition}

\begin{para}\label{P} Let $M$ be an $R$-module. For each $i \ge 0$, let $
	\beta_i(M) := \rank_k \left( \Tor^R_i(M, k) \right) $ be the $i$th \emph{Betti number} of $ M $. Set $ P_M(z) := \sum_{n \ge 0}\beta_n(M) z^n$, the \emph{Poincar\'{e} series} of $M$. The {\it complexity} of $ M $ is defined to be
	\begin{equation*}
		\cx_R(M) := \inf \left\{ b \in \mathbb{Z}_{\ge 0} \; \Big| \; \limsup_{n \to \infty} \frac{\beta_n(M)}{n^{b-1}}  < \infty \right\}.
	\end{equation*}
	It is possible that $\cx_R(M) = \infty$. For an $R$-module $M$, we have $ \cx_R(M) \le \cx_R(k) $; see \cite[4.2.4]{Avr98}. If $R$ is a complete intersection ring of codimension $c$, then it follows from \cite[Theorem~6]{Tat57} that $\cx_R(k) = c$. Furthermore, for each $i = 0,\ldots,c$, there exists an $R$-module $M_i$ such that $\cx_R(M_i) = i$.
 \end{para}
 
 \begin{para}
 	For a local ring $ (R,\mathfrak{m},k) $, Serre showed a coefficientwise inequality
 	\[
 		P_k(z) \preccurlyeq \dfrac{(1+z)^{\mu(\mathfrak{m})}}{1 - \sum_{j=1}^{\infty} \rank_k\big( H_j(K_{\bullet}) \big) z^{j+1}}
 	\]
 	of formal power series, where $ K_{\bullet} $ is the Koszul complex on a minimal set of generators of $ \mathfrak{m} $. If equality holds, then $ R $ is said to be a {\it Golod ring}. 
 \end{para}

\section{Behaviour of an Ulrich module as test module} \label{sec:Ulrich-mod-is-test-mod}
 
Here we study Ulrich modules. We start with the following theorem, which shows that every Ulrich module behaves like a test module that detects the finiteness of homological dimensions for MCM modules.
 
 \begin{theorem}\label{thm:Ulrich-mod-is-test-mod-for-mcm-mods}
 	Let $ M $ be an Ulrich $ R $-module, and $ N $ be an MCM $ R $-module.
 	\begin{enumerate}[{\rm (i)}]
 		\item If $ \Ext_R^i(M,N) = 0 $ for some $ (d+1) $ consecutive values of $ i \ge 1 $, then $ \injdim_R(N) $ is finite.
 		\item If $ \Ext_R^i(N,M) = 0 $ for some $ (d+1) $ consecutive values of $ i \ge 1 $, then $ N $ is free.
 		\item If $ \Tor_i^R(M,N) = 0 $ for some $ (d+1) $ consecutive values of $ i \ge 1 $, then $ N $ is free.
 	\end{enumerate}
 \end{theorem}
 
 \begin{proof}
 	We prove this theorem by using induction on $ d $. Let us first consider the base case $ d = 0 $. In this case, since $ M $ is Ulrich, we have $ \fm M = 0 $, i.e., $ M $ is a non-zero $ k $-vector space. Therefore $ \Ext_R^i(M,N) = 0 $ for some $ i \ge 1 $ yields that $ \Ext_R^i(k,N) = 0 $ for some $ i \ge 1 $, which implies that $ \injdim_R(N) $ is finite. For (ii) and (iii), $ \Ext_R^i(N,M) = 0 $ or $ \Tor_i^R(M,N) = 0 $ for some $ i \ge 1 $ yields that $ \Ext_R^i(N,k) = 0 $ or $ \Tor_i^R(k,N) = 0 $ for some $ i \ge 1 $, which implies that $ \projdim_R(N) $ is finite. Since $ R $ is Artinian, we obtain that $ N $ is free.
 	
 	We now give the inductive step. Assume that $ d \ge 1 $. In view of \ref{para:res-field-infinite}, we may as well assume that the residue field $ k $ is infinite. Hence there exists an $ R \oplus M \oplus N $-superficial element $ x $. Since $ \depth(R \oplus M \oplus N) = d \ge 1 $, we obtain that $ x $ is $ (R \oplus M \oplus N) $-regular. Set $\overline{(-)} := (-) \otimes_R R/(x)$. Clearly, $ \oR $ is a CM local ring of dimension $ d - 1 $, and $ \oN $ is an MCM $ \oR $-module. Moreover, $ \overline{M} $ is an Ulrich $ \overline{R} $-module by Lemma~\ref{lem:min-mult-mod-x}(ii). In view of Lemma~\ref{lem:Ext-Tor-mod-x}, since $ \Ext_R^i(M,N) = 0 $ for some $ (d+1) $ consecutive values of $ i \ge 1 $, we get that $ \Ext_{\oR}^i(\oM,\oN) = 0 $ for some $ d $ $ ( = \dim(\oR) + 1) $ consecutive values of $ i \ge 1 $. Therefore, by the induction hypothesis, we obtain that $ \injdim_{\oR}(\oN) $ is finite, which implies that $ \injdim_R(N) $ is finite. For (ii) and (iii), $ \Ext_R^i(N,M) = 0 $ (resp. $ \Tor_i^R(M,N) = 0 $) for some $ (d+1) $ consecutive values of $ i \ge 1 $ yields that $ \Ext_{\oR}^i(\oN,\oM) = 0 $ (resp. $ \Tor_i^{\oR}(\oM,\oN) = 0 $) for some $ d $ $ ( = \dim(\oR) + 1) $ consecutive values of $ i \ge 1 $. In both cases, by the induction hypothesis, we get that $ \projdim_{\oR}(\oN) $ is finite, and hence $ \projdim_R(N) $ is finite, which gives that $ N $ is free as $ N $ is MCM.
 \end{proof}
 
 As an immediate corollary of Theorem~\ref{thm:Ulrich-mod-is-test-mod-for-mcm-mods}(i), we obtain a characterization of Gorenstein local rings provided there exists an Ulrich module. The reader may compare this result with \cite[Theorems~2.2 and 2.4]{JL07}.
 
 \begin{corollary}\label{cor:charac-Gor-via-Ulrich-mod}
 	Let $ M $ be an Ulrich $ R $-module. Then $ R $ is Gorenstein if and only if $ \Ext_R^i(M,R) = 0 $ for some $ (d+1) $ consecutive values of $ i \ge 1 $.
 \end{corollary}
 
 As a consequence of Theorem~\ref{thm:Ulrich-mod-is-test-mod-for-mcm-mods}, we prove that Ulrich modules are Ext-test as well as Tor-test modules which detect the finiteness of projective dimension for arbitrary modules.
 
 \begin{corollary}\label{cor:Ulrich-mod-is-ext-tor-test-mod-projdim}
 	Suppose $ M $ and $ N $ are $ R $-modules, where $ M $ is Ulrich. Set $ t := \depth(N) $. Then the following statements hold true:
 	\begin{enumerate}[{\rm (i)}]
 		\item If $ \Ext_R^i(N,M) = 0 $ for some $ (d+1) $ consecutive values of $ i \ge d - t + 1 $, then $ \projdim_R(N) $ is finite.
 		\item If $ \Tor_i^R(N,M) = 0 $ for some $ (d+1) $ consecutive values of $ i \ge d - t + 1 $, then $ \projdim_R(N) $ is finite.
 	\end{enumerate}
 \end{corollary}
 
 \begin{proof}
 	For a short exact sequence $ 0 \to U \to V \to W \to 0 $ of $ R $-modules, by virtue of the Depth Lemma, we have $ \depth(U) \ge \min\{ \depth(V), \depth(W) + 1 \} $. Using this fact, one can prove that $ \Omega_{d-t}^R(N) $ is an MCM $ R $-module. In view of Lemma~\ref{lem:syz-ext-tor}, we get that
 	\begin{align*}
	 	\Ext_R^i(N,M) & \cong \Ext_R^{i-(d-t)} \left( \Omega_{d-t}^R(N), M \right) \quad \mbox{and }\\
	 	\Tor_i^R(N,M) & \cong \Tor_{i-(d-t)}^R \left( \Omega_{d-t}^R(N), M \right)
	\end{align*}
	for all $ i \ge d - t + 1 $. Therefore, from the hypothesis of (i) (resp. (ii)), we obtain that $ \Ext_R^j\left( \Omega_{d-t}^R(N), M \right) = 0 $ (resp. $ \Tor_j^R \left( \Omega_{d-t}^R(N), M \right) = 0 $) for some $ (d+1) $ consecutive values of $ j \ge 1 $. Hence, in either case, it follows from Theorem~\ref{thm:Ulrich-mod-is-test-mod-for-mcm-mods} that $ \Omega_{d-t}^R(N) $ is free, and hence $ \projdim_R(N) $ is finite.
 \end{proof}
 
 The following corollary shows that Ulrich modules are Ext-test modules which detect the finiteness of injective dimension for arbitrary modules.
 
 \begin{corollary}\label{cor:Ulrich-mod-is-ext-tor-test-mod-injdim}
 	Let $ M $ and $ N $ be $ R $-modules, where $ M $ is Ulrich. Let $ \Ext_R^i(M,N) = 0 $ for some $ (d+1) $ consecutive values of $ i \ge 1 $. Then $ \injdim_R(N) $ is finite.
 \end{corollary}
 
 \begin{proof}
 	We may assume that $R$ is complete.	In view of \cite[Theorem~A]{AB89}, we may consider an MCM approximation of $N$, i.e., an exact sequence $ 0 \rightarrow Y \rightarrow X \rightarrow N \rightarrow 0 $ of $ R $-modules, where $X$ is MCM and $Y$ has finite injective dimension. Since $ M $ is MCM, we have that $ \Ext^i_R(M,Y) = 0$ for every $ i \ge 1 $; see, e.g., \cite[3.1.24]{BH98}. Therefore $ \Ext^i_R(M,X) \cong \Ext^i_R(M,N) = 0 $ for some $ (d+1) $ consecutive values of $ i \ge 1 $. Since $ X $ is MCM, it follows from Theorem~\ref{thm:Ulrich-mod-is-test-mod-for-mcm-mods}(i) that $\injdim_R(X)$ is finite, and hence $\injdim_R(N)$ is finite.
 \end{proof}
 
 \begin{remark}\label{olgur-1}
 	In particular, by virtue of Corollaries~\ref{cor:Ulrich-mod-is-ext-tor-test-mod-projdim} and \ref{cor:Ulrich-mod-is-ext-tor-test-mod-injdim}, Ulrich modules belong to the following subcategories of $ \Mod(R) $ studied in \cite{CDT14}:
 	\begin{align*}
 	T(R) & := \left\{ M : \mbox{every $ N $ with }\Tor_{\gg 0}^R(M,N) = 0 \mbox{ has } \projdim_R(N) < \infty \right\},\\
 	\EP(R) & := \left\{ M : \mbox{every $ N $ with }\Ext_R^{\gg 0}(N,M) = 0 \mbox{ has } \projdim_R(N) < \infty \right\} \mbox{ and}\\
 	\EI(R) & := \left\{ M : \mbox{every $ N $ with }\Ext_R^{\gg 0}(M,N) = 0 \mbox{ has } \injdim_R(N) < \infty \right\},
 	\end{align*}
 	where $ \Mod(R) $ denotes the category of all (finitely generated) $ R $-modules. Moreover, in view of \cite[Proposition~2.7]{CDT14}, if $ R $ is a local complete intersection ring, then every Ulrich $ R $-module $ M $ has maximal complexity, i.e., $ \cx_R(M) = \codim(R) $.
 \end{remark}

\section{Vanishing of Exts and Tors over CM local rings of minimal multiplicity} \label{sec:vanishing-ext-tor}
 
 In this section, we study the vanishing of Exts or Tors over CM local rings of minimal multiplicity. We need the following well-known lemma, which shows the existence of Ulrich modules provided the base ring has minimal multiplicity. It is essentially contained in \cite[2.5]{BHU87}.
 
 \begin{lemma}\label{lem:min-mult-syz-Ulrich-mods}
 	Let $ R $ be a CM local ring of minimal multiplicity. Let $ M $ be a non-free MCM $ R $-module. Then $ \Omega_n^R(M) $ is an Ulrich $ R $-module for every $ n \ge 1 $.
 \end{lemma}
 
 Let us fix the following hypothesis for the rest of this section.
 
 \begin{hypothesis}\label{hyp:min-mult-d-s-t}
 	Let $ (R,\fm,k) $ be a $ d $-dimensional CM local ring of minimal multiplicity. Let $ M $ and $ N $ be $ R $-modules. Set $ s := \depth(M) $ and $ t := \depth(N) $.
 \end{hypothesis}
 
 The following theorem particularly shows that over a CM local ring $ R $ of minimal multiplicity, $ \Tor_i^R(M,N) = 0 $ for all $ i \gg 0 $ if and only if either $ M $ or $ N $ has finite projective dimension. It should be noted that, in \cite[Theorem~1.9]{HW97}, Huneke and Wiegand showed this result when $ R $ is a hypersurface (i.e., a regular local ring modulo a regular element).
  
 \begin{theorem}\label{thm:min-mult-tor-test}
 	Along with {\rm Hypothesis~\ref{hyp:min-mult-d-s-t}}, further assume that $ \Tor_i^R(M,N) = 0 $ for some $ (d+1) $ consecutive values of $ i \ge 2 d - (s + t) + 2 $. Then either $ \projdim_R(M) $ or $ \projdim_R(N) $ is finite.
 \end{theorem}
 
 \begin{proof}
 	By virtue of Depth Lemma, one can prove that $ \Omega_{d-s}^R(M) $ and $ \Omega_{d-t}^R(N) $ are MCM $ R $-modules. If $ \projdim_R(M) $ is finite, then there is nothing to prove. So we may assume that $ \projdim_R(M) $ is infinite. Therefore $ \Omega_{d-s}^R(M) $ is a non-free MCM $ R $-module. Hence, in view of Lemma~\ref{lem:min-mult-syz-Ulrich-mods}, we have that $ \Omega_{d-s+1}^R(M) $ is an Ulrich $ R $-module. Applying Lemma~\ref{lem:syz-ext-tor}(ii) twice, we obtain that
 	\begin{align*}
	 	\Tor_i^R(M,N) &\cong \Tor_{i-(d-s+1)}^R\left(\Omega_{d-s+1}^R(M), N\right)\\
	 	& \cong \Tor_{i-(d-s+1)-(d-t)}^R\left( \Omega_{d-s+1}^R(M), \Omega_{d-t}^R(N) \right)
 	\end{align*}
 	for all $ i-(d-s+1)-(d-t) \ge 1 $, i.e., for all $ i \ge 2 d - (s + t) + 2 $. These isomorphisms, along with the hypotheses of the theorem, yield that
 	\[
	 	\Tor_j^R\left( \Omega_{d-s+1}^R(M), \Omega_{d-t}^R(N) \right) = 0
 	\]
 	for some $ (d+1) $ consecutive values of $ j \ge 1 $. It then follows from Theorem~\ref{thm:Ulrich-mod-is-test-mod-for-mcm-mods}(iii) that $ \Omega_{d-t}^R(N) $ is free, and hence $ \projdim_R(N) $ is finite.
 \end{proof}
 
 Here we give the counterpart of Theorem~\ref{thm:min-mult-tor-test} for Ext-modules.
 
 \begin{theorem}\label{thm:min-mult-ext-test}
 	Along with {\rm Hypothesis~\ref{hyp:min-mult-d-s-t}}, further assume that $ \Ext_R^i(M,N) = 0 $ for some $ (d+1) $ consecutive values of $ i \ge d - s + 2 $. Then either $ \projdim_R(M) $ or $ \injdim_R(N) $ is finite.
 \end{theorem}
 
 \begin{proof}
 	If $ \projdim_R(M) $ is finite, then there is nothing to prove. So we may assume that $ \projdim_R(M) $ is infinite. Hence, as in the proof of Theorem~\ref{thm:min-mult-tor-test}, we get that $ \Omega_{d-s+1}^R(M) $
 	is an Ulrich $ R $-module. In view of Lemma~\ref{lem:syz-ext-tor}(i), we obtain that
 	\[
	 	\Ext_R^i(M,N) \cong \Ext_R^{i-(d-s+1)}\left( \Omega_{d-s+1}^R(M), N\right) 
 	\]
 	for all $ i-(d-s+1) \ge 1 $, i.e., for all $ i \ge d - s + 2 $.
 	These isomorphisms, along with the hypotheses of the theorem, provide that 
 	$ \Ext_R^j\left( \Omega_{d-s+1}^R(M), N\right) = 0 $ for some $ (d+1) $ consecutive values of $ j \ge 1 $. 
 	Therefore, by virtue of Corollary~\ref{cor:Ulrich-mod-is-ext-tor-test-mod-injdim}, we obtain that	$ \injdim_R(N) $ is finite.
 \end{proof}
 
 \begin{remark}\label{rmk:Saeed}
 	The authors thank Saeed Nasseh for informing them that analogous results of Theorems~\ref{thm:min-mult-tor-test} and \ref{thm:min-mult-ext-test} have been obtained in \cite[Corollaries~6.5 and 6.6]{NT}.
 	They proved these results over CM local rings with quasi-decomposable maximal ideal; see \cite[Definition~4.1]{NT}. In particular, a non-Gorenstein CM local ring with minimal multiplicity (and with infinite residue field) has quasi-decomposable maximal ideal (\cite[Example~4.7]{NT}). Although their results are more general than ours, but in the special case of rings of minimal multiplicity our results are more complete and proofs are more simple and elementary. Not only our results cover the case where $ R $ is Gorenstein with multiplicity $ e(R) = \mu(\mathfrak{m}) - d + 1 $, but also we consider the vanishing of Exts for $ i \ge d - s + 2 $, while they consider it for $ i \ge d - s + 5 $.
	For a Gorenstein local ring $ R $ with multiplicity $ e(R) = \mu(\mathfrak{m}) - d + 2 $ and $ \mu(\fm) > 2 $, it is shown in \cite[3.6 and 3.7]{HJ03} that $ \Ext^i_R(M,N) = 0 $ for all $ i \gg 0 $ if and only if either $ M $ or $ N $ has finite projective dimension.
 \end{remark}
 
 \begin{remark}\label{olgur-2}
 	If $ R $ is a CM local ring of minimal multiplicity, then $ R $ is Golod; see \cite[5.2.8]{Avr98}. It is shown in \cite[Proposition~1.4]{JS04} that for modules $ M $ and $ N $ over a Golod ring $ R $, if $ \Ext_R^i(M,N) = 0 $ for all $ i \gg 0 $, then either $ \projdim_R(M) $ or $ \injdim_R(N) $ is finite. The counterpart of this result for Tor-modules is obtained in \cite[Theorem~3.1]{Jor99a}. It should be noted that over CM local rings of minimal multiplicity, Theorems~\ref{thm:min-mult-tor-test} and \ref{thm:min-mult-ext-test} are considerably stronger which require only finitely many vanishing of Exts or Tors to detect finiteness of homological dimensions.
 \end{remark}

 The following example shows that Theorems~\ref{thm:min-mult-tor-test} and \ref{thm:min-mult-ext-test} are not necessarily true if $ R $ does not have minimal multiplicity.
 
 \begin{example}\label{exam:counter-example-if-R-is-not-min-mult}
 	Let $ R = k[X,Y]/(X^2,Y^2) $, where $ k $ is a field. Let $ x $ and $ y $ be the images of $ X $ and $ Y $ in $ R $ respectively. Then $ R $ is an Artinian local ring with the maximal ideal $ \fm := (x, y) $. Since $ \fm^2 \neq 0 $, $ R $ does not have minimal multiplicity. Set $ M := (x) $ and $ N := (y) $. Let $ E $ be the injective hull of $ k $ over $ R $. Set $ (-)^{\vee} := \Hom_R(-,E) $. Considering the minimal free resolution of $ M $:
 	\[
 		\cdots \stackrel{x}{\longrightarrow} R \stackrel{x}{\longrightarrow} R \stackrel{x}{\longrightarrow} R \to 0,
 	\]
 	one may compute that $ \Tor_i^R(M,N) = 0 $ for every $ i \ge 1 $. Hence $ \Ext_R^i(M,N^{\vee}) \cong \Tor_i^R(M,N)^{\vee} = 0 $ for every $ i \ge 1 $. Since $ M $ is annihilated by $ x $, it is not free. Similarly, $ N $ is not free. By Matlis Duality, it can be verified that $ N^{\vee} $ is not injective. 
 \end{example}
 
 The following well-known example shows that the number of consecutive vanishing of
 Tors (resp. Exts) in Theorem~\ref{thm:min-mult-tor-test} (resp. \ref{thm:min-mult-ext-test}) cannot be further reduced.
 
 \begin{example}\label{exam:rigidity-fails-for-min-mult}
 	Let $ k[[X,Y]] $ be a formal power series ring in two indeterminates $ X $ and $ Y $ over a field $ k $. Set $ R := k[[X,Y]]/(XY) $. Suppose $ x $ and $ y $ are the images of $ X $ and $ Y $ in $ R $ respectively. Set $ \fm := (x,y) $. Clearly, $ (R, \fm, k) $ is a CM local ring. It can be easily shown that $ e(R) = 2 $, $ \mu(\fm) = 2 $ and $ \dim(R) = 1 $. Therefore $ R $ has minimal multiplicity. Set $ M := (x) $, an ideal of $ R $. Note that $ M $ is an MCM $ R $-module. Considering the minimal free resolution of $ M $:
 	\[
	 	\cdots \stackrel{x\cdot}{\lra} R \stackrel{y\cdot}{\lra} R \stackrel{x\cdot}{\lra} R \stackrel{y\cdot}{\lra} R \lra 0,
 	\]
 	we can easily compute the following:
 	\begin{align*}
	 	& \Tor_{2i+1}^R(M,M) = (x)/(x^2) \neq 0 \quad \mbox{for all } i \ge 0,\\
	 	& \Tor_{2i}^R(M,M) = 0 \quad \mbox{for all } i \ge 1,\\
	 	& \Ext_R^{2i}(M,M) = (x)/(x^2) \neq 0 \quad \mbox{for all } i \ge 1 \quad \mbox{and} \\
	 	& \Ext_R^{2i+1}(M,M) = 0 \quad \mbox{for all } i \ge 0.
 	\end{align*}
 	Note that both $ \projdim_R(M) $ and $ \injdim_R(M) $ are infinite.
 \end{example}
 
 As a corollary of Theorems~\ref{thm:min-mult-tor-test} and \ref{thm:min-mult-ext-test}, we obtain a few necessary and sufficient conditions for a CM local ring of minimal multiplicity to be Gorenstein.
 
 \begin{corollary}\label{cor:Tachikawa-conj-min-mult}
 	Let $ (R,\fm,k) $ be a CM local ring of minimal multiplicity. Set $ d := \dim(R) $. Let $ \omega $ be a canonical module of $ R $. Then the following statements are equivalent:
 	\begin{enumerate}
 		\item[{\rm (i)}] $ R $ is Gorenstein.
 		\item[{\rm (ii)}] $ \Tor_i^R(\omega, \omega) = 0 $ for some $ (d+1) $ consecutive values of $ i \ge 2 $.
 		\item[{\rm (iii)}] $ \Ext_R^i(\omega, R) = 0 $ for some $ (d+1) $ consecutive values of $ i \ge 2 $.
 	\end{enumerate}
 \end{corollary}
 
 \begin{proof}
 	It is a well-known fact that $ R $ is Gorenstein if and only if $ \projdim_R(\omega) $ is finite. So the corollary follows from Theorems~\ref{thm:min-mult-tor-test} and \ref{thm:min-mult-ext-test}.
 \end{proof}
 
 \begin{remark}\label{rmk:Tachikawa-conj-min-mult}
 	In \cite{ABS05}, Avramov, Buchweitz and \c{S}ega proved in several significant
 	cases the following commutative local analog of a conjecture of Tachikawa:
 	If $ \Ext_R^i(\omega, R) = 0 $ for all $ i \ge 1 $, then $ R $ is Gorenstein.
 	In a particular case, they showed that if there is an $ R $-regular sequence
 	$ \underline{x} $ such that $ \fm^3 \subseteq (\underline{x}) $,
 	then $ \Ext_R^i(\omega, R) = 0 $ for all $ 1 \le i \le d + 1 $ implies that $ R $ is Gorenstein;	see \cite[Theorem~5.1]{ABS05}. We note that if the residue field $ k $ is infinite and $R$ has minimal multiplicity, then there is a minimal reduction
 	$J$ of $\fm$ such that $\fm^2 \subseteq J$.	The implication `(iii) $ \Rightarrow $ (i)' in Corollary~\ref{cor:Tachikawa-conj-min-mult}	does not quite follow from the result of Avramov, Buchweitz and \c{S}ega. We should also note that our proof is considerably simpler than theirs.
 	\end{remark}

\section{Criteria for regular and Gorenstein local rings via syzygy modules of the residue field}
\label{sec:criteria-reg-Gor-via-syzygy}

 In this section, we give a number of necessary and sufficient conditions for a CM local ring of minimal multiplicity to be regular or Gorenstein. These criteria are based on the vanishing of certain Exts or Tors involving syzygy modules of the residue field. Throughout this section, we are going to refer the following:
 
 \begin{hypothesis}\label{hyp:min-mult-dimR}
 	Let $ (R,\fm,k) $ be a CM local ring of minimal multiplicity. Set $ d := \dim(R) $.
 \end{hypothesis}
 
\subsection{On homomorphic images of finite direct sums of syzygy modules}\label{subsec:homomorphic-image}

 Here we consider the vanishing of Exts and Tors involving homomorphic images of finite direct sums of syzygy modules of the residue field.
 One may compare the following result with \cite[Theorem~4.1]{Gho}.
 
 \begin{proposition}\label{prop:min-mult-charc-reg-hom-image}
 	Along with {\rm Hypothesis~\ref{hyp:min-mult-dimR}}, assume that $ M $ and $ N $ are non-zero homomorphic images of finite direct sums of syzygy modules of $ k $. {\rm (}Possibly, $ M = N ${\rm )}. Set $ i_0 := 2 d - \depth(M) - \depth(N) + 2 $. Then the following statements are equivalent:
 	\begin{enumerate}
 		\item[{\rm (i)}] $ R $ is regular.
 		\item[{\rm (ii)}] $ \Tor_i^R(M,N) = 0 $ for some $ (d+1) $ consecutive values of $ i \ge i_0 $.
 	\end{enumerate}
 	Moreover, if $ N $ is MCM, then we may add the following:
 	\begin{enumerate}
 		\item[{\rm (iii)}] $ \Ext_R^i(M,N) = 0 $ for some $ (d+1) $ consecutive values of $ i \ge i_0 $.
 	\end{enumerate}
 \end{proposition}
 
 \begin{proof}
 	(i) $ \Rightarrow $ \{(ii) and (iii)\}: If $ R $ is regular, then $ \projdim_R(M) $ is finite, and hence $ \projdim_R(M) = d - \depth(M) $ (by the Auslander-Buchsbaum Formula). Therefore
 	\[
	 	\Tor_i^R(M,N) = 0 = \Ext_R^i(M,N) \quad \mbox{for all } i \ge d - \depth(M) + 1.
 	\]
 	
 	(ii) $ \Rightarrow $ (i): By virtue of Theorem~\ref{thm:min-mult-tor-test}, either $ \projdim_R(M) $ or $ \projdim_R(N) $ is finite. In either case, it follows from \cite[Proposition~7]{Mar96} that $ R $ is regular.
 	
 	(iii) $ \Rightarrow $ (i): In view of Theorem~\ref{thm:min-mult-ext-test}, either $ \projdim_R(M) $ or $ \injdim_R(N) $ is finite. If $ \projdim_R(M) $ is finite, then $ R $ is regular (due to \cite[Proposition~7]{Mar96}). In other case, we have that $ \injdim_R(N) $ is finite. Then, by virtue of \cite[Corollary~3.4]{GGP}, we get that $ R $ is regular, which completes the proof of this implication.
 \end{proof}
 
 \begin{remark}\label{rmk:disadvantage}
 	Although Proposition~\ref{prop:min-mult-charc-reg-hom-image} is stronger than the result \cite[Theorem~4.1]{Gho} in many directions, but one disadvantage is that here we consider the vanishing of $ i $th Ext or Tor for $ i \ge 2 $ at least.
 \end{remark}
 
 Here are the criteria for Gorenstein local rings. The reader may compare this result with \cite[Theorems~5.1 and 5.5]{Gho}.
 
 \begin{proposition}\label{prop:min-mult-charc-Gor-hom-image}
 	Along with {\rm Hypothesis~\ref{hyp:min-mult-dimR}}, let $ M $ be a non-zero homomorphic image of a finite direct sum of syzygy modules of $ k $. Set $ i_0 := d - \depth(M) + 2 $. Then the following statements are equivalent:
 	\begin{enumerate}
 		\item[{\rm (i)}] $ R $ is Gorenstein.
 		\item[{\rm (ii)}] $ \Ext_R^i(M,R) = 0 $ for some $ (d+1) $ consecutive values of $ i \ge i_0 $.
 	\end{enumerate}
 	Moreover, if $ R $ has a canonical module $ \omega $, then we may add the following:
 	\begin{enumerate}
 		\item[{\rm (iii)}] $ \Tor_i^R(M,\omega) = 0 $ for some $ (d+1) $ consecutive values of $ i \ge i_0 $.
 	\end{enumerate}
 \end{proposition}
 
 \begin{proof}
 	(i) $ \Rightarrow $ (ii): If $ R $ is Gorenstein, then $ \injdim_R(R) = d $ (see \cite[3.1.17]{BH98}). Hence $ \Ext_R^i(M,R) = 0 $ for all $ i \ge d + 1 $.
 	
 	(ii) $ \Rightarrow $ (i): By virtue of Theorem~\ref{thm:min-mult-ext-test}, either $ \projdim_R(M) $ or $ \injdim_R(R) $ is finite. If $ \projdim_R(M) $ is finite, then $ R $ is regular (due to \cite[Proposition~7]{Mar96}), and hence $ R $ is Gorenstein. In other case, $ \injdim_R(R) $ is finite, i.e., $ R $ is Gorenstein. So, in both cases, we obtain that $ R $ is Gorenstein.
 	
 	(i) $ \Rightarrow $ (iii): If $ R $ is Gorenstein, then $ \omega \cong R $. Hence $	\Tor_i^R(M, \omega) = 0 $ for all $ i \ge 1 $.
 	
 	(iii) $ \Rightarrow $ (i): In view of Theorem~\ref{thm:min-mult-tor-test}, either $ \projdim_R(M) $ or $ \projdim_R(\omega) $ is finite. If $ \projdim_R(M) $ is finite, then $ R $ is regular (by \cite[Proposition~7]{Mar96}), and hence $ R $ is Gorenstein. In other case, $ \projdim_R(\omega) $ is finite, which also implies that $ R $ is Gorenstein.
 \end{proof}

\subsection{On direct summands of syzygy modules}\label{subsec:dir-summ}

 We now provide a few criteria for a CM local ring of minimal multiplicity to be regular or Gorenstein in terms of direct summands of syzygy modules of the residue field. We use the following elementary result. This is probably known. But for the sake of completeness, we give its proof here.

\begin{lemma}\label{lem:finite-injdim}
	Let $ (R,\fm,k) $ be a $ d $-dimensional local ring {\rm (}not necessarily CM{\rm )}. Let $ N $ be an $ R $-module. Fix an arbitrary integer $ n \ge 1 $. Suppose $ \Ext_R^i(k,N) = 0 $ for all $ n \le i \le n + d $. Then $ \injdim_R(N) \le n - 1 $.
\end{lemma}

\begin{proof}
	We claim that $ \Ext_R^n(R/\fp,N) = 0 $ for all $ \fp \in \Spec(R) $. Fix $ \fp \in \Spec(R) $. If possible, assume that $ \Ext_R^n(R/\fp,N) \neq 0 $. Then we must have $ \fp \neq \fm $, and hence $ d \ge 1 $. Moreover, in view of \cite[3.1.13]{BH98}, there exists a prime ideal $ \fq_1 \supsetneq \fp $ such that $ \Ext_R^{n+1}(R/\fq_1,N) \neq 0 $. If $ \fq_1 = \fm $, then we are getting a contradiction. So we may assume that $ \fq_1 \neq \fm $. Then we must have $ d \ge 2 $, and there is a prime ideal $ \fq_2 \supsetneq \fq_1 $ such that $ \Ext_R^{n+2}(R/\fq_2,N) \neq 0 $ by \cite[3.1.13]{BH98}. This process must stop after some finite number of steps. That means we obtain the situation that $ \fm = \fq_r \supsetneq \fq_{r-1} \supsetneq \cdots \supsetneq \fq_1 \supsetneq \fp $ and $ \Ext_R^{n+r}(R/\fm,N) \neq 0 $ for some $ 1 \le r \le d $, which is a contradiction. Therefore $ \Ext_R^n(R/\fp,N) = 0 $ for all $ \fp \in \Spec(R) $, which implies that $ \injdim_R(N) \le n - 1 $ (see \cite[3.1.12]{BH98}).
\end{proof}

We now give the criteria for regular local rings.

\begin{theorem}\label{thm:min-mult-charc-reg-direct-summand}
	Along with {\rm Hypothesis~\ref{hyp:min-mult-dimR}}, assume that $ M $ and $ N $ are non-zero direct summands of some syzygy modules of $ k $. {\rm (}Possibly, $ M = N ${\rm )}. Then the following statements are equivalent:
	\begin{enumerate}
		\item[{\rm (i)}] $ R $ is regular.
		\item[{\rm (ii)}] $ \Tor_i^R(M,N) = 0 $ for some $ (d+1) $ consecutive values of $ i \ge 1 $.
		\item[{\rm (iii)}] $ \Ext_R^i(M,N) = 0 $ for some $ (d+1) $ consecutive values of $ i \ge 1 $.
	\end{enumerate}
\end{theorem}

\begin{proof}
	(i) $ \Rightarrow $ \{(ii) and (iii)\}: If $ R $ is regular, then $ \projdim_R(M) \le d $, and hence
	\[
	\Tor_i^R(M,N) = 0 = \Ext_R^i(M,N) \quad \mbox{for all } i \ge d + 1.
	\]
	
	\{(ii) or (iii)\} $ \Rightarrow $ (i): To prove these implications, we may without loss of generality assume that $ R $ is complete. In view of \ref{para:res-field-infinite}, we may also assume that $ k $ is infinite. To prove (ii) $ \Rightarrow $ (i) and (iii) $ \Rightarrow $ (i), we use induction on $ d $. If $ d = 0 $, then the implications follow from \cite[Theorem~4.1]{Gho} as every $ R $-module is MCM in this case. So we assume that $ d \ge 1 $, and the implications hold true for all such rings of dimension smaller than $ d $.
	
	Since the residue field of $ R $ is infinite and $ d \ge 1 $, there exists an element $ x \in \fm \smallsetminus \fm^2 $ which is both $ R $-superficial and $ R $-regular. Set $ \overline{(-)} := (-) \otimes_R R/(x) $. In view of Lemma~\ref{lem:min-mult-mod-x}(i), we have that $ \oR $ is a $ (d - 1) $-dimensional CM local ring of minimal multiplicity. Suppose that $ M $ and $ N $ are direct summands of $ \Omega_m^R(k) $ and $ \Omega_n^R(k) $ respectively for some $ m, n \ge 0 $. The following three cases may occur.
	
	{\bf Case~1.} Assume that $ m = 0 $. In this case, $ M $ must be equal to $ k $. Therefore the statement (ii) (resp. (iii)) yields that $ \Tor_i^R(k,N) = 0 $ (resp. $ \Ext_R^i(k,N) = 0 $) for some $ (d+1) $ consecutive values of $ i \ge 1 $, which gives that $ \projdim_R(N) $ is finite (resp. $ \injdim_R(N) $ is finite by Lemma~\ref{lem:finite-injdim}), and hence the implications follow from \cite[Proposition~7]{Mar96} and \cite[Theorem~3.7]{GGP} respectively.
	
	{\bf Case~2.} Assume that $ n = 0 $. In this case, $ N $ must be equal to $ k $. So the statement (ii) (resp. (iii)) yields that $ \Tor_i^R(M,k) = 0 $ (resp. $ \Ext_R^i(M,k) = 0 $) for some $ i \ge 1 $. Therefore, in either case, we obtain that $ \projdim_R(M) $ is finite, and hence $ R $ is regular by \cite[Proposition~7]{Mar96}.
	
	If none of the above two cases holds, then we must have the following:
	
	{\bf Case~3.} Assume that $ m, n \ge 1 $. In this case, since $ \Omega_m^R(k) $ and $ \Omega_n^R(k) $ are submodules of free $ R $-modules, and $ x $ is $ R $-regular, we obtain that $ x $ is regular on both $ \Omega_m^R(k) $ and $ \Omega_n^R(k) $. Hence, since $ M $ and $ N $ are direct summands of $ \Omega_m^R(k) $ and $ \Omega_n^R(k) $ respectively, $ x $ is regular on both $ M $ and $ N $ as well. Therefore, by virtue of Lemma~\ref{lem:Ext-Tor-mod-x}, the statement (ii) (resp. (iii)) yields that
	\begin{equation}\label{eqn1-reg-charc}
		\Tor_i^{\oR}(\oM,\oN) = 0 \quad (\mbox{resp. } \Ext_{\oR}^i(\oM,\oN) = 0)
	\end{equation}
	for some $ d $ consecutive values of $ i \ge 1$. Let us now fix indecomposable direct summands $ M' $ and $ N' $ of $ \oM $ and $ \oN $ respectively. Then, from \eqref{eqn1-reg-charc}, we get that
	\begin{equation}\label{eqn2-reg-charc}
		\Tor_i^{\oR}(M',N') = 0 \quad (\mbox{resp. } \Ext_{\oR}^i(M',N') = 0)
	\end{equation}
	for some $ d $ ($ = \dim(\oR) + 1 $) consecutive values of $ i \ge 1$. Since $ M $ is a direct summand of $ \Omega_m^R(k) $, we have that $ \oM $ is a direct summand of $ \overline{\Omega_m^R(k)} $. Hence $ M' $ is a direct summand of $ \overline{\Omega_m^R(k)} $. In view of \cite[Corollary~5.3]{Tak06}, we obtain the following isomorphism of $ \oR $-modules:
	\[ 
	\overline{\Omega_m^{R}(k)} \cong \Omega_m^{\overline{R}}(k) \oplus \Omega_{m-1}^{\overline{R}}(k).
	\]
	It then follows from the uniqueness of Krull-Schmidt decomposition \cite[21.35]{Lam01} that $ M' $ is isomorphic to a direct summand of $ \Omega_m^{\overline{R}}(k) $ or $ \Omega_{m-1}^{\overline{R}}(k) $. In a similar way, we get that $ N' $ is isomorphic to a direct summand of $ \Omega_n^{\overline{R}}(k) $ or $ \Omega_{n-1}^{\overline{R}}(k) $. Thus $ M' $ and $ N' $ are non-zero direct summands of some syzygy $ \oR $-modules of the residue field $ k $ of $ \oR $. Therefore, for both (ii) $ \Rightarrow $ (i) and (iii) $ \Rightarrow $ (i), in view of \eqref{eqn2-reg-charc}, by the induction hypothesis, we obtain that $ \oR $ is regular, and hence $ R $ is regular as $ x \in \fm \smallsetminus \fm^2 $ is an $ R $-regular element.
\end{proof}

\begin{remark}\label{rmk:recovering-disadvantage}
	It should be noted that in Theorem~\ref{thm:min-mult-charc-reg-direct-summand}, unlike Proposition~\ref{prop:min-mult-charc-reg-hom-image}, we consider the vanishing of $ i $th Ext or Tor for $ i \ge 1 $.
\end{remark}

\begin{remark}\label{rmk:no-of-cons-vanishing-cannot-be-reduced}
	In view of \cite[Example~4.3]{Gho}, one obtains that the number $ (d+1) $ of consecutive vanishing of Exts or Tors in Theorem~\ref{thm:min-mult-charc-reg-direct-summand} cannot be further reduced.
\end{remark}

Here we give the criteria for Gorenstein local rings.

\begin{theorem}\label{thm:min-mult-charc-Gor-direct-summand}
	Along with {\rm Hypothesis~\ref{hyp:min-mult-dimR}}, assume that $ M $ is a non-zero direct summand of some syzygy module of $ k $. Let $ \omega $ be a canonical module of $ R $. Then the following statements are equivalent:
	\begin{enumerate}[{\rm (i)}]
		\item $ R $ is Gorenstein.
		\item $ \Ext_R^i(M,R) = 0 $ for some $ (d+1) $ consecutive values of $ i \ge 1 $.
		\item $ \Tor_i^R(\omega,M) = 0 $ for some $ (d+1) $ consecutive values of $ i \ge 1 $.
		\item $ \Ext_R^i(\omega,M) = 0 $ for some $ (d+1) $ consecutive values of $ i \ge 1 $.
	\end{enumerate}
\end{theorem}

\begin{proof}
	(i) $ \Rightarrow $ (ii): If $ R $ is Gorenstein, then $ \Ext_R^i(M,R) = 0 $ for all $ i \ge d + 1 $.
	
	(i) $ \Rightarrow $ \{(iii) and (iv)\}: If $ R $ is Gorenstein, then $ \omega \cong R $, and hence
	\[
		\Tor_i^R(\omega, M) = 0 = \Ext_R^i(\omega, M) \quad\mbox{for all }i \ge 1.
	\]
	
	(ii) $ \Rightarrow $ (i), (iii) $ \Rightarrow $ (i) and (iv) $ \Rightarrow $ (i):  As before, we may without loss of generality assume that $ R $ is complete, and the residue field $ k $ is infinite. We prove these implications by using induction on $ d $. If $ d = 0 $, then these implications follow from \cite[Theorems~5.1 and 5.5]{Gho}. So we assume that $ d \ge 1 $, and these implications hold true for all such rings of dimension smaller than $ d $.
	
	Suppose that $ M $ is a direct summand of $ \Omega_m^R(k) $ for some $ m \ge 0 $.  Let us first consider the case $ m = 0 $. In this case, $ M $ must be equal to $ k $. Then the statement (ii) gives that $ \Ext_R^i(k,R) = 0 $ for some $ (d+1) $ consecutive values of $ i \ge 1 $. Hence, by Lemma~\ref{lem:finite-injdim}, we obtain that $ \injdim_R(R) $ is finite, i.e., $ R $ is Gorenstein. If $ m = 0 $, then the statement (iii) (resp. (iv)) yields that $ \Tor_i^R(\omega,k) = 0 $ (resp. $ \Ext_R^i(\omega,k) = 0 $) for some $ i \ge 1 $. In either case, we obtain that $ \projdim_R(\omega) $ is finite, which implies that $ R $ is Gorenstein. Thus all three implications hold true when $ m = 0 $. So we may assume that $ m \ge 1 $.
	
	Since the residue field of $ R $ is infinite and $ d \ge 1 $, there exists an element $ x \in \fm \smallsetminus \fm^2 $ which is both $ R $-superficial and $ R $-regular. Set $ \overline{(-)} := (-) \otimes_R R/(x) $. By Lemma~\ref{lem:min-mult-mod-x}(i), we have that $ \oR $ is a $ (d - 1) $-dimensional CM local ring of minimal multiplicity. Since $ M $ is a direct summand of $ \Omega_m^R(k) $, we get that $ \oM $ is a direct summand of $ \overline{\Omega_m^R(k)} $. We fix an indecomposable direct summand $ M' $ of $ \oM $. As in the proof of Theorem~\ref{thm:min-mult-charc-reg-direct-summand}, one obtains that
	\begin{equation}\label{eqn1-Gor-charc}
		M' \mbox{ is isomorphic to a direct summand of } \Omega_m^{\oR}(k) \mbox{ or } \Omega_{m-1}^{\oR}(k).
	\end{equation}
	Since $ x $ is $ R $-regular and $ m \ge 1 $, we get that $ x $ is $ \Omega_m^R(k) $-regular, and hence $ x $ is $ M $-regular. Since $ \omega $ is an MCM $ R $-module, $ x $ is $ \omega $-regular as well. Therefore, in view of Lemma~\ref{lem:Ext-Tor-mod-x}, the statements (ii), (iii) and (iv) yield that $ \Ext_{\oR}^i(\oM,\oR) = 0 $, $ \Tor_i^{\oR}(\overline{\omega},\oM) = 0 $ and $ \Ext_{\oR}^i(\overline{\omega},\oM) = 0 $ (respectively) for some $ d $ consecutive values of $ i \ge 1 $, which imply that
	\begin{equation}\label{eqn2-Gor-charc}
		\Ext_{\oR}^i(M',\oR) = 0, ~ \Tor_i^{\oR}(\overline{\omega},M') = 0 ~\mbox{ and }~ \Ext_{\oR}^i(\overline{\omega},M') = 0
	\end{equation}
	(respectively) for some $ d $ $ ( = \dim(\oR) + 1) $ consecutive values of $ i \ge 1 $. It is a well-known fact that $ \overline{\omega} $ is a canonical module of $ \oR $. Thus, from each of (ii), (iii) and (iv), in view of \eqref{eqn1-Gor-charc} and \eqref{eqn2-Gor-charc}, by the induction hypothesis, we obtain that $ \oR $ is Gorenstein, and hence $ R $ is Gorenstein as $ x $ is an $ R $-regular element.
\end{proof}

\section{Vanishing of Exts and Tors over deformation of CM local rings of minimal multiplicity}\label{sec:vanishing-ext-tor-deformation}
 
 Suppose $ S $ is a quotient of a $ d $-dimensional CM local ring of minimal multiplicity by a regular 
 sequence of length $ c $. Let $ M $ and $ N $ be MCM $ S $-modules. In this section, 
 it is shown that if $ \Ext_S^i(M,N) = 0 $ for some $ (d+c+1) $ consecutive values of $ i \ge 2 $, 
 then $ \Ext_S^i(M,N) = 0 $ for all $ i \ge 1 $. Moreover, if this holds true, 
 then either $ \projdim_R(M) $ or $ \injdim_R(N) $ is finite; see 
 Corollary~\ref{cor:vanshing-ext-over-deformation} for more general case when
 $ M $ and $ N $ are not necessarily MCM. We also prove that if $ \Tor_i^S(M,N) = 0 $ for some $ (d+c+1) $ consecutive
 values of $ i \ge c + 2 $, then $ \Tor_i^S(M,N) = 0 $ for all $ i \ge c + 1 $, and either
 $ \projdim_R(M) $ or $ \projdim_R(N) $ is finite; see Corollary~\ref{cor:vanshing-tor-over-deformation} for
 more general case. Let us fix the following hypothesis for this section.
 
 \begin{hypothesis}\label{hyp:deformation}
 	Let $ (R,\fm,k) $ be a $ d $-dimensional CM local ring of minimal multiplicity. Set $ S := R/(f_1,\ldots,f_c) $, where $ f_1,\ldots,f_c $ is an $ R $-regular sequence.
 \end{hypothesis}
 
 We now prove our main result of this section for Ext-modules.
 
 \begin{theorem}\label{thm:vanshing-ext-over-deformation}
 	Along with {\rm Hypothesis~\ref{hyp:deformation}}, further assume that $ M $ and $ N $ are MCM $ S $-modules. Then the following statements are equivalent:
 	\begin{enumerate}[{\rm (i)}]
 		\item $ \Ext_S^i(M,N) = 0 $ for some $ (d+c+1) $ consecutive values of $ i \ge 2 $.
 		\item $ \Ext_S^i(M,N) = 0 $ for all $ i \ge 1 $.
 	\end{enumerate}
 	Moreover, if this holds true, then either $ \projdim_R(M) $ or $ \injdim_R(N) $ is finite.
 \end{theorem}
 
 \begin{proof}
 	We may assume that $ R $ (and so $ S $) is complete. The implication (ii) $ \Rightarrow $ (i) follows trivially. So we need to prove the implication (i) $ \Rightarrow $ (ii). Suppose that $ \Ext_S^i(M,N) = 0 $ for some $ (d+c+1) $ consecutive values of $ i \ge 2 $. We show that $ \Ext_S^i(M,N) = 0 $ for all $ i \ge 1 $. Moreover, we prove that either $ \projdim_R(M) $ or $ \injdim_R(N) $ is finite. We prove these assertions by using induction on $ c $.
 	
 	Let us first consider the base case $ c = 0 $. In this case, $ S = R $. Therefore, by virtue of Theorem~\ref{thm:min-mult-ext-test}, either $ \projdim_R(M) $ or $ \injdim_R(N) $ is finite. If $ \projdim_R(M) $ is finite, then by the Auslander-Buchsbaum Formula, we get that $ M $ is a free $ R $-module, and hence $ \Ext_S^i(M,N) = 0 $ for all $ i \ge 1 $. In the other case, i.e., if $ \injdim_R(N) $ is finite, then in view of Proposition~\ref{prop:basic-facts-can-mod}(i), we have that $ N \cong \omega_R^r $ for some $ r \ge 1 $. Hence, by Proposition~\ref{prop:basic-facts-can-mod}(ii), we obtain that $ \Ext_S^i(M,N) = 0 $ for all $ i \ge 1 $. This completes the proof for the base case. We now assume that $ c \ge 1 $.
 	
 	Set $ R' := R/(f_1,\ldots,f_{c-1}) $. Clearly, $ S = R'/(f_c) $. Since $ \Ext_S^i(M,N) = 0 $ for some $ (d+c+1) $ consecutive values of $ i \ge 2 $, in view of Lemma~\ref{lem:long-exact-seq-ext}, we get that
 	\begin{equation}\label{eqn1-thm:vans-ext-over-defor}
	 	\Ext_{R'}^i(M,N) = 0 \mbox{ for some $ (d+c) $ consecutive values of } i \ge 3.
 	\end{equation}
 	Note that $ \depth_{R'}(M) = \depth_S(M) = \dim(S) = \dim(R') - 1 $. Similarly, we have that $ \depth_{R'}(N) = \dim(R') - 1 $. By virtue of \cite[Theorem~A]{AB89}, we have an MCM approximation of $ N $ as $ R' $-module:
 	\begin{equation}\label{eqn2-thm:vans-ext-over-defor}
	 	0 \lra Y \lra N' \lra N \lra 0.
 	\end{equation}
 	That is, \eqref{eqn2-thm:vans-ext-over-defor} is a short exact sequence of $ R' $-modules, where $ N' $ is an MCM $ R' $-module, and $ \injdim_{R'}(Y) $ is finite. Since $ \depth_{R'}(N) = \dim(R') - 1 $, by the Depth Lemma, $ Y $ is an MCM $ R' $-module. Therefore, in view of Proposition~\ref{prop:basic-facts-can-mod}(i), $ Y \cong \omega_{R'}^l $ for some $ l \ge 1 $. Since $ M $ is an MCM $ S $-module, we get that $ M $ is a CM $ R' $-module of dimension $ \dim(R') - 1 $. Hence, by Proposition~\ref{prop:basic-facts-can-mod}(ii), we obtain that
 	\begin{equation}\label{eqn3-thm:vans-ext-over-defor}
	 	\Ext_{R'}^i(M,Y) = 0 \quad \mbox{for all } i \neq 1.
 	\end{equation}
 	The short exact sequence \eqref{eqn2-thm:vans-ext-over-defor} yields the following long exact sequence:
 	\begin{align}\label{eqn4-thm:vans-ext-over-defor}
	 	\cdots \lra & \Ext_{R'}^i(M,Y) \lra \Ext_{R'}^i(M,N') \lra \Ext_{R'}^i(M,N) \\ \lra &\Ext_{R'}^{i+1}(M,Y) \lra \cdots.\nonumber
 	\end{align}
 	Therefore, in view of \eqref{eqn1-thm:vans-ext-over-defor} and \eqref{eqn3-thm:vans-ext-over-defor}, we get that
 	\begin{equation}\label{eqn5-thm:vans-ext-over-defor}
	 	\Ext_{R'}^i(M,N') = 0 \mbox{ for some $ (d+c) $ consecutive values of } i \ge 3.
 	\end{equation}
 	We now consider a short exact sequence of $ R' $-modules:
 	\begin{equation}\label{eqn6-thm:vans-ext-over-defor}
	 	0 \lra M' \lra F \lra M \lra 0,
 	\end{equation}
 	where $ F $ is a free $ R' $-module. Since $ \depth_{R'}(M) = \dim(R') - 1 $, by the Depth Lemma, $ M' $ is an MCM $ R' $-module. The short exact sequence \eqref{eqn6-thm:vans-ext-over-defor} yields the following long exact sequence:
 	\begin{align}\label{eqn7-thm:vans-ext-over-defor}
	 	\cdots \lra & \Ext_{R'}^i(M,N') \lra \Ext_{R'}^i(F,N') \lra \Ext_{R'}^i(M',N') \\ \lra &\Ext_{R'}^{i+1}(M,N') \lra \cdots.\nonumber
 	\end{align}
 	Note that $ \Ext_{R'}^i(F,N') = 0 $ for all $ i \ge 1 $. Hence, in view of \eqref{eqn5-thm:vans-ext-over-defor} and \eqref{eqn7-thm:vans-ext-over-defor}, we obtain that $ \Ext_{R'}^i(M',N') = 0 $ for some $ (d+(c-1)+1) $ consecutive values of $ i \ge 2 $. Therefore, since $ M' $ and $ N' $ are MCM modules over $ R' = R/(f_1,\ldots,f_{c-1}) $, by the induction hypothesis, we get that
 	\begin{equation}\label{eqn8-thm:vans-ext-over-defor}
	 	\Ext_{R'}^i(M',N') = 0 \quad \mbox{for all } i \ge 1.
 	\end{equation}
 	We also obtain that either $ \projdim_R(M') $ or $ \injdim_R(N') $ is finite.
 	
 	We now show that $ \Ext_S^i(M,N) = 0 $ for all $ i \ge 1 $. In view of \eqref{eqn7-thm:vans-ext-over-defor} and \eqref{eqn8-thm:vans-ext-over-defor}, we obtain that $ \Ext_{R'}^i(M,N') = 0 $ for all $ i \ge 2 $. Hence \eqref{eqn3-thm:vans-ext-over-defor} and \eqref{eqn4-thm:vans-ext-over-defor} yield that $ \Ext_{R'}^i(M,N) = 0 $ for all $ i \ge 2 $. Therefore, by virtue of Lemma~\ref{lem:long-exact-seq-ext}, we get that
 	\begin{equation}\label{eqn9-thm:vans-ext-over-defor}
	 	\Ext_S^i(M,N) \cong \Ext_S^{i+2}(M,N) \quad \mbox{for all } i \ge 1.
 	\end{equation}
 	Since $ \Ext_S^i(M,N) = 0 $ for some $ d+c+1 $ $ (\ge 2) $ consecutive values of $ i \ge 2 $, the isomorphisms \eqref{eqn9-thm:vans-ext-over-defor} yield that $ \Ext_S^i(M,N) = 0 $ for all $ i \ge 1 $.
 	
 	It remains to show that either $ \projdim_R(M) $ or $ \injdim_R(N) $ is finite. We show this by considering the following cases:
 	
 	{\bf Case 1.} Suppose $ \projdim_R(M') $ is finite. Then, in view of the short exact sequence \eqref{eqn6-thm:vans-ext-over-defor}, we obtain that $ \projdim_R(M) $ is finite because $ F $ is a free $ R' $-module and projective dimension of $ R' = R/(f_1,\ldots,f_{c-1}) $ as an $ R $-module is finite.
 	
 	{\bf Case 2.} Suppose $ \injdim_R(N') $ is finite. Consider the exact sequence \eqref{eqn2-thm:vans-ext-over-defor}:
 	\[
	 	0 \lra Y \lra N' \lra N \lra 0,
 	\]
 	where $ Y \cong \omega_{R'}^l $ has finite injective dimension as an $ R $-module; see Proposition~\ref{prop:basic-facts-can-mod}(iii). Therefore $ \injdim_R(N) $ is finite.
 \end{proof}
 
 As a corollary of Theorem~\ref{thm:vanshing-ext-over-deformation}, we obtain the following:
 
 \begin{corollary}\label{cor:vanshing-ext-over-deformation}
 	Along with {\rm Hypothesis~\ref{hyp:deformation}}, assume that $ M $ and $ N $ are $ S $-modules. Set $ i_0 := \dim(S) - \depth(M) $. Then the following statements are equivalent:
 	\begin{enumerate}[{\rm (i)}]
 		\item $ \Ext_S^i(M,N) = 0 $ for some $ (d+c+1) $ consecutive values of $ i \ge i_0 + 2 $.
 		\item $ \Ext_S^i(M,N) = 0 $ for all $ i \ge i_0 + 1 $.
 	\end{enumerate}
 	Moreover, if this holds true, then either $ \projdim_R(M) $ or $ \injdim_R(N) $ is finite.
 \end{corollary}
 
 \begin{proof}
 	We may assume that $R$ (and so $S$) is complete. Let $ \Ext_S^i(M,N) = 0 $ for some $ (d+c+1) $ consecutive values of $ i \ge i_0 + 2 $. We need to show that $ \Ext_S^i(M,N) = 0 $ for all $ i \ge i_0 + 1 $, and either $ \projdim_R(M) $ or $ \injdim_R(N) $ is finite.
 	
 	In view of Lemma~\ref{lem:syz-ext-tor}(i), we obtain that
 	\begin{equation}\label{eqn1-cor:vanshing-ext-over-defor}
	 	\Ext_S^i(M,N) \cong \Ext_S^{i-i_0}\left( \Omega_{i_0}^S(M), N \right) \quad \mbox{for all } i - i_0 \ge 1.
 	\end{equation}
 	So $ \Ext_S^j\left( \Omega_{i_0}^S(M), N \right) = 0 $ for some $ (d+c+1) $ consecutive values of $ j \ge 2 $. Consider an MCM approximation of $ N $, i.e., an exact sequence $ 0 \rightarrow Y \rightarrow X \rightarrow N \rightarrow 0 $ of $ S $-modules, where $X$ is MCM and $ \injdim_S(Y) $ is finite. Since $ \Omega_{i_0}^S(M) $ is an MCM $ S $-module and $ \injdim_S(Y) $ is finite, we obtain that $ \Ext_S^i\left( \Omega_{i_0}^S(M), Y \right) = 0 $ for  all $ i \ge 1 $. Therefore
 	\begin{equation}\label{eqn1.1-cor:vanshing-ext-over-defor}
	 	\Ext_S^i\left( \Omega_{i_0}^S(M), X \right) \cong \Ext_S^i\left( \Omega_{i_0}^S(M), N \right) \quad \text{for all } i \ge 1.
 	\end{equation}
 	It then follows that $ \Ext_S^j\left( \Omega_{i_0}^S(M), X \right) = 0 $ for some $ (d+c+1) $ consecutive values of $ j \ge 2 $. Hence, by virtue of Theorem~\ref{thm:vanshing-ext-over-deformation}, we get that
 	\begin{equation}\label{eqn2-cor:vanshing-ext-over-defor}
	 	\Ext_S^j\left( \Omega_{i_0}^S(M), X \right) = 0  \quad \mbox{for all } j \ge 1.
 	\end{equation}
 	Moreover, we obtain that either $ \projdim_R\left( \Omega_{i_0}^S(M) \right) $ or $ \injdim_R(X) $ is finite. It follows from \eqref{eqn1-cor:vanshing-ext-over-defor}, \eqref{eqn1.1-cor:vanshing-ext-over-defor} and \eqref{eqn2-cor:vanshing-ext-over-defor} that $ \Ext_S^i(M,N) = 0 $ for all $ i \ge i_0 + 1 $. Note that if $ \projdim_R\left( \Omega_{i_0}^S(M) \right) $ is finite, then $ \projdim_R(M) $ is finite because projective dimension of $ S = R/(f_1,\ldots,f_c) $ as an $ R $-module is finite. Since $ \injdim_S(Y) $ is finite, $Y$ has a finite resolution by $\omega_S$; see, e.g., \cite[3.3.28(b)]{BH98}. But, in view of Proposition~\ref{prop:basic-facts-can-mod}(iii), $\omega_S = \omega_R/(f_1, \ldots, f_c)\omega_R$ has finite injective dimension as an $ R $-module. Therefore $ \injdim_R(Y) $ is finite, which yields that $ \injdim_R(N) $ is finite provided $ \injdim_R(X) $ is finite. This completes the proof of the corollary.
 \end{proof}
 
 Here we prove our main result of this section for Tor-modules.
 
 \begin{theorem}\label{thm:vanshing-tor-over-deformation}
 	Along with {\rm Hypothesis~\ref{hyp:deformation}}, further assume that $ M $ and $ N $ are MCM $ S $-modules. Then the following statements are equivalent:
 	\begin{enumerate}[{\rm (i)}]
 		\item $ \Tor_i^S(M,N) = 0 $ for some $ (d+c+1) $ consecutive values of $ i \ge c + 2 $.
 		\item $ \Tor_i^S(M,N) = 0 $ for all $ i \ge c + 1 $.
 	\end{enumerate}
 	Moreover, if this holds true, then either $ \projdim_R(M) $ or $ \projdim_R(N) $ is finite.
 \end{theorem}
 
 \begin{proof}
 	We may assume that $ R $ (and so $ S $) is complete. The implication (ii) $ \Rightarrow $ (i) follows trivially. So we need to prove the implication (i) $ \Rightarrow $ (ii). Suppose that $ \Tor_i^S(M,N) = 0 $ for some $ (d+c+1) $ consecutive values of $ i \ge c + 2 $. We show that $ \Tor_i^S(M,N) = 0 $ for all $ i \ge c + 1 $. Moreover, we prove that either $ \projdim_R(M) $ or $ \projdim_R(N) $ is finite. To prove these assertions, as in the proof of Theorem~\ref{thm:vanshing-ext-over-deformation}, we use induction on $ c $.
 	
 	We first consider the base case $ c = 0 $. In this case, $ S = R $. Therefore, by virtue of Theorem~\ref{thm:min-mult-tor-test}, either $ \projdim_R(M) $ or $ \projdim_R(N) $ is finite. If $ \projdim_R(M) $ is finite, then by the Auslander-Buchsbaum Formula, $ M $ is a free $ R $-module, and hence $ \Tor_i^S(M,N) = 0 $ for all $ i \ge 1 $. In another case, i.e., if $ \projdim_R(N) $ is finite, then $ N $ is a free $ R $-module. In this case also, $ \Tor_i^S(M,N) = 0 $ for all $ i \ge 1 $. This completes the proof for the base case $ c = 0 $. We now assume that $ c \ge 1 $.
 	
 	Set $ R' := R/(f_1,\ldots,f_{c-1}) $. Clearly, $ S = R'/(f_c) $. Since $ \Tor_i^S(M,N) = 0 $ for some $ (d+c+1) $ consecutive values of $ i \ge c + 2 $, in view of Lemma~\ref{lem:long-exact-seq-tor}, we get that
 	\begin{equation}\label{eqn1-thm:vans-tor-over-defor}
	 	\Tor_i^{R'}(M,N) = 0 \mbox{ for some $ (d+c) $ consecutive values of } i \ge c + 3.
 	\end{equation}
 	Note that $ \depth_{R'}(M) = \depth_S(M) = \dim(S) = \dim(R') - 1 $. Similarly, one obtains that $ \depth_{R'}(N) = \dim(R') - 1 $. Consider the following short exact sequences of $ R' $-modules:
 	\begin{equation}\label{eqn2-thm:vans-tor-over-defor}
	 	0 \to M' \to F \to M \to 0 \quad \mbox{ and }\quad 0 \to N' \to G \to N \to 0,
 	\end{equation}
 	where $ F $ and $ G $ are free $ R' $-modules. Clearly, by the Depth Lemma, $ M' $ and $ N' $ are MCM $ R' $-modules. The short exact sequences \eqref{eqn2-thm:vans-tor-over-defor} yield the following long exact sequences:
 	\begin{align}
	 	\cdots \lra & \Tor_{i+1}^{R'}(M',N) \lra \Tor_{i+1}^{R'}(F,N) \lra \Tor_{i+1}^{R'}(M,N)\label{eqn3-thm:vans-tor-over-defor}\\
	 	\lra &\Tor_i^{R'}(M',N) \lra \cdots\nonumber\\
	 	& \quad \quad \quad \quad \quad \quad \quad \quad \mbox{ and}\nonumber\\
	 	\cdots \lra & \Tor_{i+1}^{R'}(M',N') \lra \Tor_{i+1}^{R'}(M',G) \lra \Tor_{i+1}^{R'}(M',N)\label{eqn4-thm:vans-tor-over-defor}\\
	 	\lra &\Tor_i^{R'}(M',N') \lra \cdots\nonumber
 	\end{align}
 	respectively. Note that $ \Tor_i^{R'}(F,N) = 0 = \Tor_i^{R'}(M',G) $ for all $ i \ge 1 $. Therefore, in view of \eqref{eqn1-thm:vans-tor-over-defor} and \eqref{eqn3-thm:vans-tor-over-defor}, we obtain that
 	\begin{equation}\label{eqn5-thm:vans-tor-over-defor}
	 	\Tor_i^{R'}(M',N) = 0 \mbox{ for some $ (d+c) $ consecutive values of } i \ge c + 2.
 	\end{equation}
 	Hence \eqref{eqn4-thm:vans-tor-over-defor} and \eqref{eqn5-thm:vans-tor-over-defor} yield that $ \Tor_i^{R'}(M',N') = 0 $ for some $ ( d+(c-1)+1 ) $ consecutive values of $ i \ge c + 1 $ $ (= (c-1) + 2) $. 	Therefore, since $ M' $ and $ N' $ are MCM modules over $ R' = R/(f_1,\ldots,f_{c-1}) $, by the induction hypothesis, we get that
 	\begin{equation}\label{eqn6-thm:vans-tor-over-defor}
	 	\Tor_i^{R'}(M',N') = 0 \quad \mbox{for all } i \ge (c-1) + 1 ~( = c).
 	\end{equation}
 	We also obtain that either $ \projdim_R(M') $ or $ \projdim_R(N') $ is finite. Hence, in view of the short exact sequences \eqref{eqn2-thm:vans-tor-over-defor}, we get that either $ \projdim_R(M) $ or $ \projdim_R(N) $ is finite (because $ F $ and $ G $ are free $ R' $-modules and projective dimension of $ R' = R/(f_1,\ldots,f_{c-1}) $ as an $ R $-module is finite).
 	
 	It remains to show that $ \Tor_i^S(M,N) = 0 $ for all $ i \ge c + 1 $. In view of \eqref{eqn4-thm:vans-tor-over-defor} and \eqref{eqn6-thm:vans-tor-over-defor}, we obtain that $ \Tor_i^{R'}(M',N) = 0 $ for all $ i \ge c+1 $. Therefore \eqref{eqn3-thm:vans-tor-over-defor} yields that $ \Tor_i^{R'}(M,N) = 0 $ for all $ i \ge c + 2 $. Hence, by virtue of Lemma~\ref{lem:long-exact-seq-tor},
 	\begin{equation}\label{eqn7-thm:vans-tor-over-defor}
	 	\Tor_i^S(M,N) \cong \Tor_{i+2}^S(M,N) \quad \mbox{for all } i \ge c + 1.
 	\end{equation}
 	Since $ \Tor_i^S(M,N) = 0 $ for some $ d+c+1 $ $ (\ge 2) $ consecutive values of $ i \ge c + 2 $, the isomorphisms \eqref{eqn7-thm:vans-tor-over-defor} yield that $ \Tor_i^S(M,N) = 0 $ for all $ i \ge c + 1 $. This completes the proof of the theorem.
 \end{proof}
 
 As a corollary of Theorem~\ref{thm:vanshing-tor-over-deformation}, we obtain the following:
 
 \begin{corollary}\label{cor:vanshing-tor-over-deformation}
 	Along with {\rm Hypothesis~\ref{hyp:deformation}}, further assume that $ M $ and $ N $ are $ S $-modules. Set $ i_0 := 2 \dim(S) - \depth(M) - \depth(N) $. Then the following statements are equivalent:
 	\begin{enumerate}[{\rm (i)}]
 		\item $ \Tor_i^S(M,N) = 0 $ for some $ (d+c+1) $ consecutive values of $ i \ge i_0 + c + 2 $.
 		\item $ \Tor_i^S(M,N) = 0 $ for all $ i \ge i_0 + c + 1 $.
 	\end{enumerate}
 	Moreover, if this holds true, then either $ \projdim_R(M) $ or $ \projdim_R(N) $ is finite.
 \end{corollary}
 
 \begin{proof}
 	Let $ \Tor_i^S(M,N) = 0 $ for some $ (d+c+1) $ consecutive values of $ i \ge i_0 + c + 2 $. Set $ c_M := \dim(S) - \depth(M) $ and $ c_N := \dim(S) - \depth(N) $. Then $ i_0 = c_M + c_N $. In view of Lemma~\ref{lem:syz-ext-tor}(ii), we obtain that
 	\begin{equation}\label{eqn1-cor:vanshing-tor-over-defor}
	 	\Tor_i^S(M,N) \cong \Tor_{i-c_M}^S\left( \Omega_{c_M}^S(M), N \right)
	 	\cong \Tor_{i - i_0}^S\left( \Omega_{c_M}^S(M), \Omega_{c_N}^S(N) \right)
 	\end{equation}
 	for all $ i - i_0 \ge 1 $. Therefore $ \Tor_j^S\left( \Omega_{c_M}^S(M), \Omega_{c_N}^S(N) \right) = 0 $ for some $ (d+c+1) $ consecutive values of $ j \ge c + 2 $. Since $ \Omega_{c_M}^S(M) $ and $ \Omega_{c_N}^S(N) $ are MCM $ S $-modules, by virtue of Theorem~\ref{thm:vanshing-tor-over-deformation}, we get that
 	\begin{equation}\label{eqn2-cor:vanshing-tor-over-defor}
	 	\Tor_j^S\left( \Omega_{c_M}^S(M), \Omega_{c_N}^S(N) \right) = 0  \quad \mbox{for all } j \ge c + 1.
 	\end{equation}
 	We also obtain that either $ \projdim_R\left( \Omega_{c_M}^S(M) \right) $ or $ \projdim_R\left( \Omega_{c_N}^S(N) \right) $ is finite, which yields that either $ \projdim_R(M) $ or $ \projdim_R(N) $ is finite because projective dimension of $ S = R/(f_1,\ldots,f_c) $ as an $ R $-module is finite. It follows from \eqref{eqn1-cor:vanshing-tor-over-defor} and \eqref{eqn2-cor:vanshing-tor-over-defor} that $ \Tor_i^S(M,N) = 0 $ for all $ i \ge i_0 + c + 1 $.
 \end{proof}
 
 \begin{remark}
 	With the hypotheses as in Corollary~\ref{cor:vanshing-tor-over-deformation}, $ \Tor_i^S(M,N) = 0 $ 
 	for all $ i \gg 0 $ does not necessarily imply that either $ \projdim_S(M) $ or $ \projdim_S(N) $ is finite, 
 	due to an example of Jorgensen \cite[4.2]{Jor99b}, where $ S $ is even a local complete intersection 
 	ring of codimension $ 2 $. Then, by virtue of \cite[Remark~6.3]{AB00}, we observe that
 	the same example works for Ext-modules also, i.e., $ \Ext_S^i(M,N) = 0 $ for all $ i \gg 0 $ 
 	does not necessarily imply that either  $ \projdim_S(M) $ or $ \injdim_S(N) $ is finite.
 	However, if $ R $ is a non-Gorenstein CM local ring with minimal multiplicity and infinite residue field, by \cite[3.1 and 4.7]{NT}, there is an $ R $-regular sequence $ \underline{y} $ such that $ S := R/(\underline{y}) $ is a fiber product. Then, by \cite[Corollaries~6.2 and 6.3]{NT}, the vanishing of $ \Tor^S_i(M,N) $ (resp. $ \Ext_S^i(M,N) $) for all $ i \gg 0 $ implies the finiteness of projective or injective dimensions of the modules $ M $ and $ N $ over $ S $.
 \end{remark}
 
 \section{Applications}\label{sec:application}
 
 In this section, we assume the following:
 
\begin{hypothesis}\label{hyp:deformation-app}
 	Let $ (R,\fm,k) $ be a $ d $-dimensional CM local ring of minimal multiplicity.
 	Set $ S := R/(f_1,\ldots,f_c) $, where $ f_1,\ldots,f_c $ is an $ R $-regular sequence. Also assume that $R$ has a canonical module $\omega_R$.
 	So $\omega_S = \omega_R/(f_1, \ldots, f_c)\omega_R$ is a canonical module of $S$.
 \end{hypothesis}

Our first application is that conjecture of Tachikawa holds true for $ S $. In particular, we prove the following:

\begin{theorem}\label{Tachikawa}
	Along with {\rm Hypothesis~\ref{hyp:deformation-app}}, if $\Ext^i_S(\omega_S, S) = 0 $ for all $i \gg 0$, then $S$ is Gorenstein.
\end{theorem}

\begin{proof}
 By virtue of Theorem~\ref{thm:vanshing-ext-over-deformation}, either $\projdim_R(\omega_S) $ or $ \injdim_R(S) $ is finite. Therefore, in view of Lemma~\ref{rach}, we get that either $\projdim_R(\omega_R) $ or $\injdim_R(R)$ is finite. In both cases, $R$ is Gorenstein, and hence $ S $ is Gorenstein.
\end{proof}

The following results are well-known and easy to prove.

\begin{lemma}\label{rach}
 Let $ A $ be a local ring, and $M$ be an $A$-module. Let $ x \in A $ be an $ M $-regular element. Then the following statements hold true.
 \begin{enumerate}[\rm (i)]
 	\item $\injdim_A(M)$ is finite if and only if $\injdim_A(M/xM)$ is finite.
 	\item $\projdim_A(M)$ is finite if and only if $\projdim_A(M/xM)$ is finite.
 \end{enumerate}
\end{lemma}

As another application, we obtain the following:

\begin{theorem}\label{app2}
	Along with {\rm Hypothesis~\ref{hyp:deformation}}, assume that $ M $ and $ N $ are non-zero homomorphic images of finite direct sums of syzygy modules of $ k $ over $ S $. {\rm (}Possibly, $ M = N ${\rm )}. If $\Tor^S_i(M, N) = 0$ for all $i \gg 0$, then $S$ is regular.
\end{theorem}

\begin{proof}
	Note that $ \cx_S(M) = \cx_S(N) = \cx_S(k) $ (due to \cite[Corollary~9]{Avr96}). In view of Corollary~\ref{cor:vanshing-tor-over-deformation}, either $ \projdim_R(M) $ or $ \projdim_R(N) $ is finite. Suppose $ \projdim_R(M) $ is finite. Hence, by virtue of \cite[Theorem~3.1]{Gul74}, we obtain that $ \cx_S(M) $ is finite. Therefore $ \cx_S(k) $ is finite, and hence $ S $ is a complete intersection ring (by \cite[(2.3)]{Gul80}). Suppose codimension of $S$ is $l$. Then $ \cx_S(M) = \cx_S(N) = \cx_S(k) = l $. Since $\Tor^S_i(M, N) = 0$ for all $i \gg 0$, in view of \cite[9.3.9]{Avr98}, it follows that $ l = 0 $, and hence $ S $ is regular.
\end{proof}

Our final application is the following:

\begin{theorem}\label{app3}
	Along with {\rm Hypothesis~\ref{hyp:deformation-app}}, let $ M $ be a non-zero homomorphic image of a finite direct sum of syzygy modules of $ k $ over $ S $. Then the following statements are equivalent:
	\begin{enumerate}[{\rm (i)}]
		\item $ S $ is Gorenstein.
		\item $ \Ext_S^i(M,S) = 0 $ for all $i \gg 0$.
		\item $ \Tor_i^S(M,\omega_S) = 0 $ for all $i \gg 0$.
	\end{enumerate}
\end{theorem}

\begin{proof}
	As in the proof of Theorem~\ref{app2}, we note that if $ \projdim_R(M) $ is finite, then $ S $ is a complete intersection ring, and hence $ S $ is Gorenstein.
	
	(ii) $ \Rightarrow $ (i): Suppose $ \Ext_S^i(M,S) = 0 $ for all $i \gg 0$. Then, by virtue of Corollary~\ref{cor:vanshing-ext-over-deformation}, either $ \projdim_R(M) $ or $ \injdim_R(S) $ is finite. If $ \projdim_R(M) $ is finite, then $ S $ is Gorenstein. In other case, i.e., if $ \injdim_R(S) $ is finite, then by Lemma~\ref{rach}, $ \injdim_R(R) $ is finite, i.e., $R$ is Gorenstein, and hence $ S $ is Gorenstein.
	
	(iii) $ \Rightarrow $ (i): Suppose $ \Tor_i^S(M,\omega_S) = 0 $ for all $i \gg 0$. 	Then, in view of Corollary~\ref{cor:vanshing-tor-over-deformation}, either $ \projdim_R(M) $ or $ \projdim_R(\omega_S) $ is finite. In either case, by a similar way as above, one obtains that $ S $ is Gorenstein.
\end{proof}

\section*{Acknowledgements}

The authors would like to thank Olgur Celikbas for his valuable comments on this article and particularly for making Remarks~\ref{olgur-1} and \ref{olgur-2}. The authors also thank the anonymous reviewer for his/her careful reading of the manuscript and many valuable comments and suggestions.

\end{document}